\documentclass[10pt,a4paper,oneside]{amsproc}

\usepackage{latexsym, amsthm}
\usepackage{amsfonts, amsmath, amssymb,comment}
\usepackage{euscript,mathrsfs}
\usepackage{url}
\usepackage{array}
\usepackage[a4paper]{geometry}
\usepackage{graphics}
\usepackage[usenames]{color}
\usepackage[all]{xy}
\usepackage{graphicx}
\usepackage{boxedminipage}

\newtheorem{theorem}{Theorem}[section]

\newtheorem{conjecture}{Conjecture}
\newtheorem{corollary}[theorem]{Corollary}
\newtheorem{question}[theorem]{Question}
\newtheorem{definition}[theorem]{Definition}
\newtheorem{remark}[theorem]{Remark}

\newtheorem{lemma}[theorem]{Lemma}

\newtheorem{proposition}[theorem]{Proposition}

\newcommand{\codim}{\mathrm{codim}}
\newcommand{\B}{\mathcal{B}}

\newcommand{\Fqt}{\mathbb{F}_{q^2}}
\newcommand{\Fq}{\mathbb{F}_q}

\newcommand{\C}{\mathcal{C}}
\newcommand{\Q}{\mathcal{Q}}
\newcommand{\I}{\mathcal{I}}
\newcommand{\U}{\mathcal{U}}

\def\Fq{{\mathbb F}_q}
\def\AA{{\mathbb A}}
\def\FF{{\mathbb F}}
\def\PP{{\mathbb P}}

\newcommand{\X}{\mathcal{X}}

\begin{document}

\title[Intersection of non-degenerate Hermitian varieties and cubic hypersurfaces]{Maximum number of points of intersection of a non-degenerate Hermitian variety and a cubic hypersurface}

\author{Subrata Manna}
\address{Department of Mathematics, \newline \indent
Indian Institute of Technology Hyderabad, Kandi, Sangareddy, Telangana, India}
\email{subrata147314@gmail.com}
\thanks{The author is partially supported by a doctoral fellowship from the Council of Scientific and Industrial Research, Govt. of India}


\begin{abstract}   
 Edoukou, Ling and Xing in 2010, conjectured that in $\PP^n(\Fqt)$, $n\geq 3$, the maximum number of common points of a non-degenerate Hermitian variety $\U_n$ and a hypersurface of degree $d$ is achieved only when the hypersurface is union of $d$ distinct hyperplanes meeting in a common linear space $\Pi_{n-2}$ of codimension $2$ such that $\Pi_{n-2}\cap \U_n$ is a non-degenerate Hermitian variety. Furthermore, these $d$ hyperplanes are tangent to $\U_n$ if $n$ is odd and non-tangent if $n$ is even. In this paper, we show that the conjecture is true for $d=3$ and $q\geq 7$.

\end{abstract}

\date{}
\maketitle

\section{Introduction}
Hermitian varieties over finite fields are among the most widely studied varieties in algebraic geometry due to their alluring geometric structure, high number of rational points, and applications to coding theory. These varieties became more prominent among coding theorists and combinatorists after G. Lachaud introduced functional codes in \cite{L}. Although Hermitian varieties were studied by Dickson from a group-theoretic point of view in \cite{Di}, Bose and Chakravarti were the first to study these varieties with a strong emphasis on geometry in \cite{BC}. In order to determine the minimum distance and the minimum weight codewords of the functional code defined over a non-degenerate Hermitian variety as well as from a combinatorial perspective, the following question has been studied in \cite{BBFS, BD, BDH, DM, ELX, HS} among other articles:

\begin{question}\label{question}
    Let $\U_n$ be a non-degenerate Hermitian variety in $\PP^n(\Fqt)$. Determine the maximum number of $\Fqt$-rational points common to a hypersurface $V(F)$ of degree $d$ in $\PP^n$ defined over $\Fqt$ and $\U_n$. Furthermore, determine the structure of the hypersurface $V(F)$ attaining the maximum number of points.
\end{question}

Bose and Chakravarti answered the above question for $d=1$ in \cite{BC, C}. When $n=2$, it follows from B\'{e}zout's theorem that a curve of degree $d\leq q$ and a non-degenerate Hermitian curve (which is essentially an irreducible curve) can intersect at a maximum of $d(q+1)$ many $\Fqt$-rational points. Although it is easy to set an example of a curve of degree $d$ defined over $\Fqt$ which intersects the non-degenerate Hermitian curve at exactly $d(q+1)$ many $\Fqt$-rational points, a complete characterization of curves defined over $\Fqt$ achieving $d(q+1)$ many $\Fqt$-rational points common to a non-degenerate Hermitian curve is still unresolved. To this end, a few examples of such curves can be found in \cite{BDMN}.

For the case when $n=3$, S{\o}rensen, in his Ph.D. thesis \cite{SoT} conjectured that a non-degenerate Hermitian surface $\U_3$ in $\PP^3(\Fqt)$ has at most $d(q^3+q^2-q)+q+1$ many $\Fqt$-rational points in common with a surface of degree $d\leq q$ defined over $\Fqt$ and the maximum is attained by a surface if and only if the surface is the union of $d$ planes in $\PP^3(\Fqt)$ that are tangent to $\U_3$, each containing a common line $\ell$ intersecting $\U_3$ at $q+1$ points. S{\o}rensen's conjecture was proved for $d=2$ in \cite{E}, for $d=3$ in \cite{BD}, and finally, the conjecture was settled for any $d\leq q$ in \cite{BDH}.

Now, it is natural to look at the Question \ref{question} when $n\geq 4$. To this end, Edoukou, Ling, and Xing in \cite{ELX} posed a conjecture which can be rephrased as follows:

\begin{conjecture}\cite[Conjecture 2 (ii)]{ELX}\label{ELXC}
   Let $n\geq 3$ and $\U_n$ be a non-degenerate Hermitian variety in $\PP^n(\Fqt)$. If $V(F)$ is a hypersurface of degree $d\leq q$, then $\max_{V(F)} |V(F)(\Fqt)\cap \U_n|$ is attained if and only if the hypersurface $V(F)$ is union of $d$ distinct hyperplanes $\Sigma_1, \Sigma_2,\dots, \Sigma_d$ such that each of the hyperplanes contains a common $(n-2)$-dimensional linear space $\Pi$  and $\Pi\cap \U_n$ is a non-degenerate Hermitian variety and: 

     --- If $n$ is even, the $d$ hyperplanes $\Sigma_1,\dots, \Sigma_d$ are non-tangent to $\U_n$.

     --- If $n$ is odd, the $d$ hyperplanes $\Sigma_1,\dots, \Sigma_d$ are tangent to $\U_n$. 
\end{conjecture}

As already mentioned above, the Conjecture \ref{ELXC} is true for $n=3$. For $n\geq 4$, Hallez and Storme proved this conjecture in \cite{HS}, for $d=2$ under the constraint that $n<O(q^2)$. Bartoli, Boeck, Fanali, and Storme later proved the conjecture for any integer $n\geq 4$ and $d=2$ in \cite[Theorem 3.3]{BBFS}. Notably, the proof of Hallez and Storme, as well as that of Bartoli et al., relies on the classification of quadrics in projective space over a finite field. However, we will be more algebro-geometric towards our approach, though the key idea remains aligned with that of  \cite{HS, BBFS}. 

In this paper, we investigate the Conjecture~\ref{ELXC} for cubic hypersurfaces, i.e., the case $d=3$ with $n \geq 4$, and establish an affirmative answer when $q \geq 7$. More precisely, we prove the following.

\begin{theorem}[Theorem~\ref{evendim} and Theorem  \ref{odddim}]\label{Final}

   Let $n\geq 4$ and $\U_n$ be a non-degenerate Hermitian variety in $\PP^n(\Fqt)$. If $V(F)$ is a cubic hypersurface in $\PP^n$ defined over $\Fqt$, then for $q\geq 7$,
   $$\max_{V(F)} |V(F)(\Fqt)\cap \U_n|=\begin{cases}
          3|\U_{n-1}(\Fqt)|-2|\U_{n-2}(\Fqt)|\ \ \text{if $n$ is even} \\
          (3q^2-2)|\U_{n-2}(\Fqt)|+3\ \ \text{if $n$ is odd}.
       \end{cases}$$
Moreover, this maximum is attained if and only if the cubic hypersurface \( V(F) \) satisfies the conditions stated in the Conjecture \ref{ELXC}.       
\end{theorem}

The rest of this paper is organized as follows. In Section \ref{sec:prel}, we revisit various well-known properties of Hermitian varieties with a particular focus on non-degenerate  Hermitian varieties, review some results related to the quadric sections of a non-degenerate Hermitian variety, and some preliminary notions from algebraic geometry. In Section \ref{sec: cubic section} we describe the cubic sections of a non-degenerate Hermitian variety. Finally, in Section \ref{sec: structure}, we explore the structure of a cubic hypersurface sharing the maximum number of points with a non-degenerate Hermitian variety in both even and odd dimensions.

\section{Preliminaries}\label{sec:prel}
We fix an integer $q$, a prime power. As customary, $\Fq$ and $\Fqt$ denote the finite fields with $q$ and $q^2$ elements, respectively. For $n \ge 0$, we denote by $\PP^n$, the projective space of dimension $n$ over the algebraic closure $\overline{\FF}_q$, while $\PP^n (\Fqt)$ will denote the set of all $\Fqt$-rational points on $\PP^n$. 
Likewise, $\AA^n$ and $\AA^n (\Fqt)$ will denote the affine space of dimension $n$ over $\overline{\FF}_q$ and the set of all $\Fqt$-rational points of $\AA^n$ respectively.  
Furthermore, for a homogeneous polynomial $F \in \Fqt[x_0, \dots , x_n]$, we denote by $V(F)$, the set of zeroes of $F$ in $\PP^n$ and by $V(F)(\Fqt)$ the set of all $\Fqt$-rational points of $V(F)$. By an algebraic variety, we mean a set of zeroes of a certain set of polynomials in the affine space or projective space, depending on the context. In particular, an algebraic variety need not be irreducible.  However, when we say that a hypersurface $V(F)$ is irreducible, we will mean that the polynomial $F$ is irreducible in the field of its definition.

This section is organized into three subsections. The first subsection reviews several well-known facts about Hermitian varieties over finite fields. The second subsection focuses on the intersection of a non-degenerate Hermitian variety with a quadric. The third subsection revisits some fundamental concepts from basic algebraic geometry. All the results presented in this section are already established, and their proofs can be found in the cited references.

\subsection{Geometry of Hermitian varieties over a finite field}\label{Geometry of Hermitian}
First, we recall the definition and key properties of Hermitian varieties (cf. \cite{BC, C}), which will be essential in the later sections of this paper.

\begin{definition}[\cite{BC}] \normalfont
Let $H = (h_{ij})$ ($0 \le i, j \le n$) be an $(n+1) \times (n+1)$ matrix with entries in $\Fqt$. We denote by $H^{(q)}$ the matrix whose $(i, j)$-th entry is given by $h_{ij}^q$. The matrix $H$ is said to be a \textit{Hermitian matrix} if $H \neq 0$ and $H^T = H^{(q)}$. A \textit{Hermitian variety} of dimension $n-1$, denoted by $\U_n$, is the zero set of the polynomial $x^T H x^{(q)}$ inside $\PP^n$, where $H$ is an $(n+1) \times (n+1)$ Hermitian matrix and $x = (x_0, \dots, x_n)^T$. The Hermitian variety is classified as \textit{non-degenerate} if $\mathrm{rank} \ H = n+1$ and \textit{degenerate} otherwise.
\end{definition}

Bose and Chakravarti in \cite[Corollary of Theorem 4.1]{BC} proved that if the rank of a Hermitian matrix is $r$, then after a suitable projective change of coordinates over $\Fqt$, we can recognize the corresponding Hermitian variety by the zero set of the polynomial
\begin{equation}\label{reduceherm}
x_0^{q+1} + x_1^{q+1} + \dots + x_{r-1}^{q+1} = 0.
\end{equation}
Here we note that the polynomial $x_0^{q+1} + x_1^{q+1} + \dots + x_{r-1}^{q+1}$ is irreducible over $\overline{\FF}_q$ only when $r\geq 3$. Thus, it follows that a Hermitian variety of rank at least three is always absolutely irreducible. Moreover, the number of $\Fqt$-rational points of a non-degenerate Hermitian variety $\U_n$ is also well known and well established in the literature.
\begin{theorem}\cite[Theorem 8.1]{BC}\label{nondeg point}\ Let $\U_n$ be a non-degenerate Hermitian variety in $\PP^n$. Then $$|\U_n(\Fqt)|=\frac{\big(q^n-(-1)^n\big)\left(q^{n+1}-(-1)^{n+1}\right)}{q^2-1}.$$
\end{theorem}

\noindent As mentioned in the introduction, the hyperplane sections of a non-degenerate Hermitian variety are well-studied, thanks to \cite{BC, C}.

\begin{theorem}\label{hyperplane section}
     Let $\U_n$ be a non-degenerate Hermitian variety in $\PP^n$ and $\Sigma$ be a hyperplane of $\PP^{n}$ defined over $\Fqt$. 
     \begin{enumerate}
     \item[(a)] \cite[Theorem 7.4]{BC} if $\Sigma$ is tangent to $\U_n$ at a point $P\in \U_n(\Fqt)$, then $\Sigma \cap \U_n$ is a degenerate Hermitian variety of rank $n-1$ contained in $\Sigma$. That is, 
     $\Sigma \cap \U_n$ is a cone over a non-degenerate Hermitian variety $\U_{n-2}$ contained in a hyperplane of $\Sigma$ with center at $P$. 
     \item[(b)] \cite[Theorem 3.1]{C} if $\Sigma$ is not a tangent to $\U_n$, then $\Sigma \cap \U_n$ is a non-degenerate Hermitian variety $\U_{n-1}$ in $\Sigma$.
     \end{enumerate}
     
 \end{theorem} 

\begin{remark}\label{polar}\normalfont
   Chakravarti proved \cite[Theorem 3.1]{C} for a polar hyperplane at an external point of $\PP^n(\Fqt)$ with respect to $\U_n=V\left(x^THx^{(q)}\right)$, where a point $C\in\PP^n(\Fqt)$ with row vector $c^T=(c_0,c_1,\dots,c_n)$ is called an external point with respect to $\U_n$ if $c^THc^{(q)}\neq 0$ and the polar hyperplane $\Sigma$ of $\U_n$ at $C$ is the set of $\Fqt$-rational solutions of $x^THc^{(q)}=0$. If $\Sigma_0:=V\left(\sum_{i=0}^na_ix_i\right)$ is a hyperplane in $\PP^n$ defined over $\Fqt$, that is not a tangent hyperplane to $\U_n$, it follows that $P:=\left(H^{-1}(a_0,a_1,\dots,a_n)^T\right)^{(q)}$ is an external point in $\PP^n(\Fqt)$ with respect to $\U_n$ and $\Sigma_0$ is the polar hyperplane at $P$, otherwise the tangent hyperplane to $\U_n$ at $P$ is $\Sigma_0$.
\end{remark}

\begin{lemma}\label{Hermitian}
   Let $\Sigma$ be a hyperplane in $\PP^n$ defined over $\Fqt$ and $\U_n$ be a non-degenerate Hermitian variety. Then $$|\Sigma(\Fqt)\cap \U_n|\leq\begin{cases}
       |\U_{n-1}(\Fqt)|\ \ \text{if $n$ is even}\\
       q^2|\U_{n-2}(\Fqt)|+1\ \ \text{if $n$ is odd.}
   \end{cases}$$
\end{lemma}
\begin{proof}
    The Lemma is a direct consequence of Theorem \ref{nondeg point} and  Theorem \ref{hyperplane section}.
\end{proof}
We now revisit the intersection of a linear subspace with a non-degenerate Hermitian variety and observe its possible structures. Throughout this article, $\Pi_r \U_s$ denotes the \textit{cone} in $\PP^n$ with vertex the linear subspace $\Pi_r \subset \PP^n$ and base the non-degenerate Hermitian variety $\U_s \subset \Pi_s \cong \PP^s$, where $\Pi_r$ and $\Pi_s$ are disjoint linear subspaces of $\PP^n$, of dimensions $r$ and $s$, respectively, and together span $\PP^n$.

\begin{theorem}\cite[Lemma 2.20]{HT}\label{linear subspace section}
  Let $\Pi_m$ be a linear subspace of dimension $m$ in $\PP^n$ defined over $\Fqt$. Then there is a section $\Pi_m\cap \U_n=\Pi_v\U_s$ if and only if $$T:=n-2m+s\geq 0,$$ where $-1\leq v\leq m,\ -1\leq s\leq m$ and $s+v=m-1$. The section is non-degenerate when $v=-1$, and $\Pi_m$ lies entirely on $\U_n$ when $s=-1$.   
\end{theorem}

\begin{corollary}\label{codim two intersection}
 Let $\Pi_{n-2}$ be a linear subspace of codimension two in $\PP^n$ defined over $\Fqt$. Then there are three types of sections of $\U_n$ with $\Pi_{n-2}:$ $$ 
\Pi_{n-2}\cap \U_n=
\begin{cases}
 \U_{n-2}\\
 \Pi_0\U_{n-3}\\
 \Pi_1 \U_{n-4},
\end{cases}
$$ where $\Pi_0$ is an $\Fqt$-rational point of $\PP^n$ and $\Pi_1$ is a line in $\PP^n$ defined over $\Fqt$.   
\end{corollary}
\begin{proof}
    The proof is a straightforward consequence of Theorem \ref{linear subspace section}.
\end{proof}

\begin{remark}\label{cardinality of codim two}
\normalfont Geometrically, $\Pi_0\U_{n-3}$ is a cone over the non-degenerate Hermitian variety $\U_{n-3}$ with vertex at $\Pi_0$. Thus it follows from Theorem \ref{nondeg point} that $$|\Pi_0\U_{n-3}(\Fqt)|=1+q^2\frac{\big(q^{n-3}-(-1)^{n-3}\big)\left(q^{n-2}-(-1)^{n-2}\right)}{q^2-1}.$$ And since $\Pi_1\U_{n-4}$ is a cone with vertex as the line $\Pi_1$ and base as $\U_{n-4}$, therefore 
    $$|\Pi_1\U_{n-4}(\Fqt)|=q^2+1+q^4|\U_{n-4}(\Fqt)|.$$ For ease of computation, let us list the cardinalities of three types of sections of $\U_n$ with $\Pi_{n-2}$ explicitly.
    
   \textbf{Case 1:} \textit{$n$ is even.}
 \begin{align*}
   |\U_{n-2}(\Fqt)|&=\frac{q^{2n-3}-q^{n-1}+q^{n-2}-1}{q^2-1}\\
    |\Pi_0\U_{n-3}(\Fqt)|&=\frac{q^{2n-3}+q^n-q^{n-1}-1}{q^2-1}\\
    |\Pi_1 \U_{n-4}(\Fqt)|&=\frac{q^{2n-3}-q^{n+1}+q^n-1}{q^2-1}.
    \end{align*}

 \textbf{Case 2:} \textit{$n$ is odd.}
 \begin{align*}
    |\U_{n-2}(\Fqt)|&=\frac{q^{2n-3}+q^{n-1}-q^{n-2}-1}{q^2-1}\\
    |\Pi_0\U_{n-3}(\Fqt)|&=\frac{q^{2n-3}-q^n+q^{n-1}-1}{q^2-1}\\
    |\Pi_1 \U_{n-4}(\Fqt)|&=\frac{q^{2n-3}+q^{n+1}-q^n-1}{q^2-1}.
     \end{align*} 
    
    \noindent Therefore when $n$ is an even integer a straightway computation shows $$|\Pi_1 \U_{n-4}(\Fqt)|<|\U_{n-2}(\Fqt)|<|\Pi_0\U_{n-3}(\Fqt)|,$$ and when $n$ is odd then $$|\Pi_0\U_{n-3}(\Fqt)|< |\U_{n-2}(\Fqt)|<|\Pi_1\U_{n-4}(\Fqt)|.$$

\end{remark}

\subsection{Intersection of a non-degenerate Hermitian variety with a quadric}\label{intersection with quadric}

In this subsection, we will revisit some results concerning the quadric section of a non-degenerate Hermitian variety from \cite{BBFS}, which will be instrumental in later discussions.
\begin{definition}\cite[Definition 3.1]{BBFS}\label{An}
   Let $n$ be a positive integer. We define $A_4:=q^5+q^4+4q^3-3q+1$ and for $n\geq 5$ 
  $$ A_n=\begin{cases}
    q^2A_{n-1}-q^{n-2} \ \ \text{if} \ n \ \text{is even}\\
    q^2A_{n-1}+q^{n-2}+2q^{n-3}\ \ \text{if} \ n \ \text{is odd.}
\end{cases}$$
\end{definition}

The lemma below, which can be proved using induction on $n$, presents Definition \ref{An} in a more concise form.

\begin{lemma}\cite[Lemma 3.2]{BBFS}\label{Anconcise}
    For $n\geq 4$ we have $$A_n=q^{2n-8}A_4+\sum_{i=n-2}^{2n-7}q^i+2\delta q^{n-3},$$ with $\delta=0$ if $n$ is even and $\delta=1$ if $n$ is odd.
\end{lemma}
 The next theorem is the key to prove the Conjecture \ref{ELXC} for the quadric case.
\begin{theorem}\cite[Theorem 3.3]{BBFS}
 Let $\Q$ be a quadric in $\PP^n$ defined over $\Fqt$ with $n\geq 4$. If $$|\Q(\Fqt)\cap \U_n|>A_n,$$ then $\Q$ is union of two hyperplanes defined over $\Fqt$.   
\end{theorem}

\subsection{Preliminaries from algebraic geometry}\label{Preli}

Here, we recall several fundamental results from algebraic geometry that will play a crucial role in subsequent sections. We will rely on concepts such as the dimension and degree of variety, as discussed in standard algebraic geometry textbooks, including Harris \cite{H}. We begin with the following definitions. A variety is said to have \textit{pure dimension} or be \textit{equidimensional} if all its irreducible components share the same dimension. Additionally, two equidimensional varieties $X,Y\subset \PP^n$ are said to \textit{intersect properly} if $\codim(X\cap Y)=\codim X+\codim Y$.
Notably, obtaining effective upper bounds for the number of rational points on varieties defined over a finite field depends on the degree and dimension of the variety. For this reason, the following Proposition from \cite{H} will be indispensable.

\begin{proposition}\cite[Cor. 18.5]{H}\label{deg}
Let $X$ and $Y$ be equidimensional varieties that intersect properly in $\PP^n$. Then 
$\deg (X \cap Y) \le \deg X \deg Y$.
    \end{proposition}
In particular, we have the following Proposition relating to the intersection of two hypersurfaces.

\begin{proposition}\label{coprime}
Let $F, G \in \Fqt[x_0, x_1, \dots, x_n]$ be non-constant homogeneous polynomials having no common factors. Then 
\begin{enumerate}
\item[(a)] $V(F)$ and $V(G)$ intersect properly. 
\item[(b)] $V(F) \cap V(G)$ is equidimensional of dimension $n-2$.
\item[(c)] $\deg (V(F) \cap V(G)) \le \deg F \deg G$. 
\end{enumerate}   
\end{proposition}
\begin{proof}
    Part (a) is a well-known result in algebraic geometry and is typically established using Krull's principal ideal theorem. Part (b) is derived from Macaulay's unmixedness theorem (see \cite[Chapter 7, Theorem 26]{ZS}). Part (c) follows directly from Proposition \ref{deg}.
\end{proof}
We conclude this subsection by stating a result of Lachaud and Rolland on the number of rational points of a variety.
\begin{proposition}\cite[Prop. 2.3]{LR}\label{lac}
Let $X$ be an equidimensional projective (resp. affine) variety defined over a finite field $\Fq$. Furthermore, assume that $\dim X = \delta$ and $\deg X = d$.  Then
$$|X(\Fq)| \le d p_{\delta} \ \ \ \ \ \ (\mathrm{resp.} \ \  |X (\Fq)| \le d q^{\delta}),$$
where $p_{\delta} = 1 + q + \dots + q^{\delta}$.
\end{proposition}

\section{Cubic section of a non-degenerate Hermitian variety}\label{sec: cubic section}
In this section, we introduce a sequence $\{B_n\}_{n\geq 4}$
  that will play an indispensable role in our analysis of intersections of Hermitian varieties and cubic hypersurfaces. This sequence is defined recursively, capturing structural properties that will be essential in the subsequent lemmas and proofs. The definition is as follows.
\begin{definition}\label{variables}\normalfont
Let $n$ be a positive integer. We define $B_4:=3(q^5+1)$ and for $n\geq 5$
$$ B_n:=\begin{cases}
    q^2B_{n-1}-q^{n-2} \ \ \text{if} \ n \ \text{is even}\\
    q^2B_{n-1}+3q^{n-2}+q^{n-3}\ \ \text{if} \ n \ \text{is odd.}
\end{cases}$$
\end{definition}

\begin{lemma}\label{form}
   For $n\geq 4$, 
   $$B_n=\begin{cases}
       3q^{2n-8}(q^5+1)+3\left(\sum_{i=1}^\frac{n-4}{2}q^{2i+n-3}\right) \ \ \text{if} \ n  \ \text{is even}\\
       
       3q^{2n-8}(q^5+1)+3\left(\sum_{i=1}^\frac{n-3}{2}q^{2i+n-4}\right)+q^{n-3} \ \ \text{if}  \ n  \ \text{is odd.}
   \end{cases}$$ 
   \end{lemma}
   \begin{proof}
       Clearly, the assertion is true for $n=4,5$. Suppose the assertion is true for $n-1$. We distinguish two cases:

        {\bf Case 1:} \textit{n is even}. It follows from Definition \ref{variables}, that $B_n=q^2B_{n-1}-q^{n-2}$. Therefore, 
        \begin{align*}
            B_n &=q^2 \left[3q^{2n-10}(q^5+1)+3\sum_{i=1}^\frac{n-4}{2}q^{2i+n-5}+q^{n-4}\right]-q^{n-2} \ \ \text{(by induction hypothesis)}\\
                &=3q^{2n-8}(q^5+1)+3\left(\sum_{i=1}^\frac{n-4}{2}q^{2i+n-3}\right)+q^{n-2}-q^{n-2}\\
                &=3q^{2n-8}(q^5+1)+ 3\left(\sum_{i=1}^\frac{n-4}{2}q^{2i+n-3}\right).
        \end{align*}
        
        {\bf Case 2:} \textit{$n$ is odd.} Then by Definition \ref{variables}, $B_n=q^2B_{n-1}+3q^{n-2}+q^{n-3}$. Hence,
        \begin{align*}
            B_n &=q^2 \left[3q^{2n-10}(q^5+1)+3\sum_{i=1}^\frac{n-5}{2}q^{2i+n-4}\right]+3q^{n-2}+q^{n-3} \ \  \text{(using induction hypothesis)}\\
                &=3q^{2n-8}(q^5+1)+3\left(\sum_{i=1}^\frac{n-5}{2}q^{2i+n-2}\right)+3q^{n-2}+q^{n-3}\\
                &=3q^{2n-8}(q^5+1)+ 3\left(\sum_{i=1}^\frac{n-3}{2}q^{2i+n-4}\right)+q^{n-3}.
        \end{align*}
      The assertion now follows from induction. 
   \end{proof}
   
   \noindent The following corollary is a consequence of Lemma \ref{form}.
\begin{corollary}\label{inq}
 Let $n\geq 5$. Then $$B_{n-1}\begin{cases}
     > q^{2n-5}+q^{2n-6} \ \ \text{if $n$ is even}\\
     < 3q^{2n-5}+q^{2n-6}\ \ \text{if $n$ is odd.}
 \end{cases}$$
\end{corollary}

 \begin{proof}
    The proof is divided into two cases:

     {\bf Case 1:} \textit{$n$ is even}. From Lemma \ref{form} it follows that 
     \begin{align*}
         B_{n-1} &= q^{2n-10}B_4+3\left(\sum_{i=1}^\frac{n-4}{2}q^{2i+n-5}\right)+q^{n-4}\\
                &= 3q^{2n-5}+3q^{2n-10}+3\left(\sum_{i=1}^\frac{n-4}{2}q^{2i+n-5}\right)+q^{n-4}\\
                &> q^{2n-5}+q^{2n-6}.
     \end{align*}
     The last inequality follows since $3q^{2n-5}>q^{2n-5}+q^{2n-6}$.

        {\bf Case 2:} \textit{$n$ is odd}. It follows from Lemma \ref{form} that 
        \begin{align*}
            (3q^{2n-5}+q^{2n-6})-B_{n-1}&=(3q^{2n-5}+q^{2n-6})-q^{2n-10}B_4-3\left(\sum_{i=1}^\frac{n-5}{2}q^{2i+n-4}\right)\\
                 &=q^{2n-6}-3q^{2n-10}-3\left(\frac{q^{2n-7}-q^{n-2}}{q^2-1}\right)\\
                 &=\frac{1}{q^2-1}\Bigg[q^{2n-10}(q^6-q^4-3q^3-3q^2+3)+3q^{n-2}\Bigg]\\
                 &>0.
        \end{align*}
        The last inequality follows as $q^6-q^4-3q^3-3q^2+3>0$. 
         \end{proof}

\begin{lemma}\label{ineqality1}
    Let $\Sigma$ be a hyperplane in $\PP^n$ defined over $\Fqt$, where $n\geq 4$. Then $$|\Sigma(\Fqt)\cap \U_n|+A_n<B_n,$$ for $q\geq 3$.
\end{lemma}
    \begin{proof}
      We consider two cases:
      
     {\bf Case 1:} \textit{$n$ is even}. In this case Lemma \ref{Hermitian} entails that $|\Sigma(\Fqt)\cap \U_n|\leq |\U_{n-1}(\Fqt)|=\frac{1}{q^2-1}(q^{2n-1}+q^n-q^{n-1}-1)$. Therefore using the expressions of $A_n$ and $B_n$ from Lemma \ref{Anconcise} and from Lemma \ref{form} respectively, we conclude 
     \begin{align*}
         B_n-A_n-|\Sigma(\Fqt)\cap \U_n| 
         &\geq q^{2n-8}(B_4-A_4)+3\left(\sum_{i=1}^\frac{n-4}{2}q^{2i+n-3}\right)-\sum_{i=n-2}^{2n-7}q^i \\
        & \hspace{3cm}  -\frac{q^{2n-1}+q^n-q^{n-1}-1}{q^2-1}\\
        &=q^{2n-8}(2q^5-q^4-4q^3+3q+2)+3\left(\frac{q^{2n-5}-q^{n-1}}{q^2-1}\right)\\
        &\  \  \ \  -\frac{(q+1)(q^{2n-6}-q^{n-2})}{q^2-1}-\frac{q^{2n-1}+q^n-q^{n-1}-1}{q^2-1}\\
        &=\frac{1}{q^2-1}\Big[q^{2n-8}\left(q^7-q^6-6q^5+q^4+9q^3+q^2-3q-2\right)\\
       &\hspace{3cm} -q^{n-2}(q^2+q-1)+1\Big]\\
       &>0.
     \end{align*}
    
The above inequality is easy to check for $n=4$, and for $n\geq 6$ it follows since $$(q^7-q^6-6q^5+q^4+9q^3+q^2-3q-2)-(q^2+q-1)>0, \ \ \text{for $q\geq 3$}.$$

     {\bf Case 2:} \textit{$n$ is odd}. From Lemma \ref{Hermitian} it follows that $|\Sigma(\Fqt)\cap \U_n|\leq q^2|\U_{n-2}(\Fqt)|+1=\frac{q^{2n-1}+q^{n+1}-q^n-1}{q^2-1}$. Consequently, 
     \begin{align*}
          B_n-A_n-|\Sigma(\Fqt)\cap \U_n|
          &\geq q^{2n-8}(B_4-A_4)+3\left(\sum_{i=1}^\frac{n-3}{2}q^{2i+n-4}\right)+q^{n-3}-\sum_{i=n-2}^{2n-7}q^i \\
        & \hspace{3cm} -2q^{n-3} -\frac{q^{2n-1}+q^{n+1}-q^{n}-1}{q^2-1}\\
        &=q^{2n-8}(2q^5-q^4-4q^3+3q+2)+3\left(\frac{q^{2n-5}-q^{n-2}}{q^2-1}\right)-q^{n-3}\\
        &\  \  \ \  -\frac{(q+1)(q^{2n-6}-q^{n-2})}{q^2-1}-\frac{q^{2n-1}+q^{n+1}-q^n-1}{q^2-1}\\
        &=\frac{1}{q^2-1}\Big[q^{2n-8}(q^7-q^6-6q^5+q^4+9q^3+q^2-3q-2)\\
        &\hspace{3cm} -q^{n-3}(q^4-q^3+2q-1)+1\Big]\\
        & >0.
     \end{align*}
     The last inequality follows since  
     \begin{align*}
     &(q^7-q^6-6q^5+q^4+9q^3+q^2-3q-2)-(q^4-q^3+2q-1)\\
     &=q^7-q^6-6q^5+10q^3+q^2-5q-1\\
     &>0,\ \ \text{for $q\geq 3$.}
     \end{align*} Combining the two cases together, the lemma follows. 
    \end{proof}

\begin{theorem}\label{casefour}
Let $q\geq 7$ and $\C_4$ be a cubic threefold in $\PP^4$ defined over $\Fqt$. If $$|\C_4(\Fqt)\cap \U_4|> B_4,$$ then $\C_4$ contains a hyperplane in $\PP^4$ defined over $\Fqt$. 
\end{theorem}
\begin{proof}
Let us suppose that $\C_4$ does not contain any hyperplane defined over $\Fqt$. If $\C_4$ contains no generator of $\U_4$ or, if $\C_4$ contains a plane defined over $\Fqt$, then by \cite[Proposition 4.1 and Proposition 4.2]{DM} we conclude that $|\C_4(\Fqt)\cap \U_4|\leq 3(q^5+1)$. Thus, we may assume that $\C_4$ contains a generator of $\U_4$, but no plane in $\PP^4$, defined over $\Fqt$. If $\C_4$ contains a tangent line to $\U_4$ then \cite[Proposition 4.8]{DM} entails that $|\C_4(\Fqt)\cap \U_4|\leq 2q^5+2q^4+3q^3+2q^2+1\leq 3(q^5+1)$, as $q\geq 7$. Therefore for any plane $\Pi$ in $\PP^4$ defined over $\Fqt$, the cubic plane curve $\C_4\cap \Pi$ contains at most three generators of $\U_4$. If there exists a plane $\Pi$ in $\PP^4$ defined over $\Fqt$ such that $\C_4\cap \Pi$ contains exactly three generators, then \cite[Proposition 4.3]{DM} implies that $|\C_4(\Fqt)\cap \U_4|\leq 3q^5-q^4+3q^3+3q^2+1\leq 3(q^5+1)$, since $q\geq 7$. Furthermore, for $q\geq 7$, if for any plane $\Pi$ defined over $\Fqt$, the cubic curve $\C_4(\Fqt)\cap \Pi$ contains at most two generators of $\U_4$, then from the proof of \cite[Proposition 5.6]{DM} we observe that $|\C_4(\Fqt)\cap \U_4|\leq 2q^5+5q^4-q^3+q^2-3q-4\leq 3(q^5+1)$. Thus if $\C_4$ contains no hyperplane in $\PP^4$ defined over $\Fqt$, then $|\C_4(\Fqt)\cap \U_4|\leq 3(q^5+1)=B_4$.  
\end{proof}

\begin{lemma}\label{numberoftangent}
 Let $P$ be an $\Fqt$-rational point of $\U_n$. The number of tangent hyperplanes to $\U_n$ passing through $P$ and defined over $\Fqt$ is $q^2|\U_{n-2}(\Fqt)|+1.$   
\end{lemma}
\begin{proof}
  For any point $Q\in \U_n(\Fqt)$, let $T_Q(\U_n)$ denote the tangent hyperplane to $\U_n$ at $Q$. Now consider the set $$\Gamma:=\{Q\in \U_n(\Fqt) :\ P\in T_Q(\U_n)\}.$$ 

 \textbf{Claim:} $\Gamma=T_P(\U_n)\cap \U_n(\Fqt)$.

 \textit{Proof of claim:} For any point $Q\in \U_n(\Fqt)$, it follows from Theorem \ref{hyperplane section} (a) that $T_Q(\U_n)\cap \U_n(\Fqt)$ is a cone over a non-degenerate Hermitian variety $\U_{n-2}$ with vertex at $Q$ and hence  $Q\in \Gamma$ entails that the line $PQ\subset \U_n$. Since any line $\ell$ with $P\in \ell \subset \U_n$ is contained in $T_P(\U_n)$, so the line $PQ\subset T_P(\U_n)$ and consequently $Q \in T_P(\U_n)\cap \U_n(\Fqt)$. Similarly, if $Q\in T_P(\U_n)\cap \U_n(\Fqt)$ then $PQ\subset \U_n$ and hence $P\in T_Q(\U_n)$. Finally, the Lemma follows from Theorem \ref{hyperplane section} (a). 
\end{proof}

\begin{proposition}\label{nontangent}
    Let $n\geq 5$ and $\C_n$ be a cubic hypersurface in $\PP^n$ defined over $\Fqt$. If $|\C_n(\Fqt)\cap \U_n|>B_n$, then there exists a hyperplane $\Sigma_0$ in $\PP^n$ defined over $\Fqt$ such that $\Sigma_0$ is non-tangent to $\U_n$ and $|\C_n(\Fqt)\cap \Sigma_0\cap \U_n|> B_{n-1}$. 
\end{proposition}
\begin{proof}
     Suppose that for every non-tangent hyperplane $\Sigma$ to $\U_n$ the following inequality is satisfied $$|\C_n(\Fqt)\cap \Sigma \cap \U_n|\leq B_{n-1}.$$
Let us consider the incidence set: 
$$\I=\{(P,\Sigma): P\in \C_n(\Fqt)\cap \U_n, \Sigma\subset \PP^n\ \text{is a $\Fqt$-hyperplane non-tangent to $\U_n$}, P\in \Sigma\}.$$
We will count $|\I|$ in two ways. First, by Lemma~\ref{numberoftangent}, the number of tangent hyperplanes to $\U_n$ passing through an $\Fqt$-rational point of $\U_n$ is independent of the chosen point, so we see that
\begin{align*}
|\I|&=\sum_{P\in \C_n(\Fqt)\cap \U_n}\#\{\Sigma: P\in \Sigma\ \ \text{and $\Sigma$ is non-tangent to $\U_n$ } \}\\
    &= |\C_n(\Fqt)\cap \U_n|\bigg[\#\{\Sigma: P\in \Sigma \ , \Sigma \ \text{is a hyperplane}\}-\\
    & \hspace{4cm}\#\{\Sigma: P\in \Sigma \ , \Sigma \ \text{is a tangent hyperplane to $\U_n$}\}\bigg] \\
    &=|\C_n(\Fqt)\cap \U_n|\left[ \frac{q^{2n}-1}{q^2-1}-q^2|\U_{n-2}(\Fqt)|-1\right]\ \ (\text{using Lemma \ref{numberoftangent}})\\
    &>B_n\frac{q^{2n}-q^2-q^2(q^{n-2}-(-1)^{n-2})(q^{n-1}-(-1)^{n-1})}{q^2-1}.
\end{align*}
    On the other hand,
    \begin{align*}
        |\I|&=\sum_{{\substack{\Sigma, \\ \text{$\Sigma$ is non-tangent}\\ \text{to $\U_n$}}} }|\C_n(\Fqt)\cap \Sigma\cap \U_n|\\
        &\leq \left(|\PP^n(\Fqt)|-|\U_n(\Fqt)|\right) B_{n-1}\\
        &= \left[\frac{q^{2n+2}-1-(q^n-(-1)^n)(q^{n+1}-(-1)^{n+1})}{q^2-1}\right]B_{n-1}.
    \end{align*}
    Thus, comparing the two bounds on $\I$, we get
    \begin{equation}\label{double}
         B_n<\left[\frac{q^{2n+2}-1-(q^n-(-1)^n)(q^{n+1}-(-1)^{n+1})}{q^{2n}-q^2-q^2(q^{n-2}-(-1)^{n-2})(q^{n-1}-(-1)^{n-1})}\right]B_{n-1}.
    \end{equation}
    We now distinguish two cases:

    {\bf Case 1:} \textit{$n$ is even}. From Equation \eqref{double} we obtain
    \begin{align*}
        B_n &< \left[\frac{q^{2n+2}-1-(q^n-1)(q^{n+1}+1)}{q^{2n}-q^2-q^2(q^{n-2}-1)(q^{n-1}+1)}\right]B_{n-1}\\
        &=B_{n-1}q^2-B_{n-1}\frac{q^{n+3}-q^{n+2}-q^{n+1}+q^n}{q^{2n}-q^{2n-1}+q^{n+1}-q^n}\\
        &<B_{n-1}q^2-(q^{2n-5}+q^{2n-6})\frac{q^{n+3}-q^{n+2}-q^{n+1}+q^n}{q^{2n}-q^{2n-1}+q^{n+1}-q^n}\ \ (\text{using Corollary \ref{inq}})\\
        &=B_{n-1}q^2-\frac{q^{3n-2}-2q^{3n-4}+q^{3n-6}}{q^{2n}-q^{2n-1}+q^{n+1}-q^n}\\
        &=B_{n-1}q^2-q^{n-2}-\frac{q^{3n-3}-2q^{3n-4}+q^{3n-6}-q^{2n-1}+q^{2n-2}}{q^{2n}-q^{2n-1}+q^{n+1}-q^n}\\
        &<B_{n-1}q^2-q^{n-2}=B_n,
    \end{align*}
    which leads us to a contradiction.

    {\bf Case 2:} \textit{$n$ is odd}. Equation \eqref{double} entails that
    \begin{align*}
     B_n &< \left[\frac{q^{2n+2}-1-(q^n+1)(q^{n+1}-1)}{q^{2n}-q^2-q^2(q^{n-2}+1)(q^{n-1}-1)}\right]B_{n-1} \\  
     & =q^2B_{n-1}+B_{n-1}\frac{q^{n+3}-q^{n+2}-q^{n+1}+q^n}{q^{2n}-q^{2n-1}-q^{n+1}+q^n}\\
     &<q^2B_{n-1}+(3q^{2n-5}+q^{2n-6})\frac{q^{n+3}-q^{n+2}-q^{n+1}+q^n}{q^{2n}-q^{2n-1}-q^{n+1}+q^n}\ \ (\text{using Corollary \ref{inq}})\\
     &=q^2B_{n-1}+3\left[\frac{q^{3n-2}-q^{3n-3}-q^{3n-4}+q^{3n-5}}{q^{2n}-q^{2n-1}-q^{n+1}+q^n}\right]+\frac{q^{3n-3}-q^{3n-4}-q^{3n-5}+q^{3n-6}}{q^{2n}-q^{2n-1}-q^{n+1}+q^n}\\
     &=q^2B_{n-1}+3q^{n-2}+q^{n-3}-\frac{3q^{3n-4}-2q^{3n-5}-q^{3n-6}-3q^{2n-1}+2q^{2n-2}+q^{2n-3}}{q^{2n}-q^{2n-1}-q^{n+1}+q^n}\\
     &=q^2B_{n-1}+3q^{n-2}+q^{n-3}-\frac{(3q^2-2q-1)(q^{3n-6}-q^{2n-3})}{q^{2n}-q^{2n-1}-q^{n+1}+q^n}\\
     &<q^2B_{n-1}+3q^{n-2}+q^{n-3}=B_n,
    \end{align*}
    which is a contradiction to the definition of $B_n$. Thus, in both cases, we arrive at a contradiction. 
\end{proof}

\begin{lemma}\label{AffineLachaud}
 Let $\X$ be a hypersurface of degree $d$ in $\PP^n$ defined over $\Fqt$ with $d\leq q$ and $n\geq 3$. If $\Sigma$ is a hyperplane defined over $\Fqt$ such that $\Sigma\nsubseteq \X$, then $$|\X(\Fqt)\cap \U_n\cap (\Sigma\setminus \Pi)|\leq (d-1)(q+1)q^{2n-6},$$ for any $(n-2)$-dimensional subspace $\Pi$ of $\Sigma$ contained in $\X$.
\end{lemma}
\begin{proof}
  Since $\Sigma \nsubseteq \X$, therefore $\Sigma\cap \X$ is a hypersurface of degree $d$ in $\Sigma\cong \PP^{n-1}$. Moreover, since $\Pi\subseteq \Sigma\cap \X$, we conclude that $(\Sigma\cap \X)\setminus \Pi$ is an affine hypersurface of degree at most $d-1$ in $\Sigma\setminus \Pi\cong \AA^{n-1}$. Furthermore, since $n\geq 3$ it follows that $\U_n\cap (\Sigma\setminus \Pi)$ is an irreducible affine hypersurface of degree $q+1$. As $d\leq q$, so $(\Sigma\cap \X)\setminus \Pi$ and $\U_n\cap (\Sigma\setminus \Pi)$ have no common components and consequently, their intersection is a complete intersection of degree at most $(d-1)(q+1)$ and $\dim \left(\X(\Fqt)\cap \U_n\cap (\Sigma\setminus \Pi)\right)=n-3$. Thus it follows from Proposition \ref{lac} that $|\X(\Fqt)\cap \U_n\cap (\Sigma\setminus \Pi)|\leq (d-1)(q+1)q^{2n-6}$.
\end{proof}

\begin{proposition}\label{hyperplane}
 Let $\C_n$ be a cubic hypersurface in $\PP^n$ defined over $\Fqt$ with $n\geq 4$ and $q\geq 7$. If $|\C_n(\Fqt)\cap \U_n|>B_n$, then $\C_n$ contains a hyperplane defined over $\Fqt$.   
\end{proposition}

\begin{proof}
    We will prove the Proposition using induction on $n$. It follows from Theorem \ref{casefour} that the Proposition is true for $n=4$.  Suppose that the Proposition is true for $n-1$. It follows from the Proposition \ref{nontangent} that there exists a non-tangent hyperplane $\Sigma_0$ to $\U_n$ such that $|\C_n(\Fqt)\cap\Sigma_0\cap \U_n|>B_{n-1}$. Now, if $\Sigma_0\subseteq \C_n$, we are through. So we may suppose that $\Sigma_0\nsubseteq \C_n$. Therefore $\C_n\cap\Sigma_0$ is a cubic hypersurface, say $\C_{n-1}$ in $\Sigma_0\cong \PP^{n-1}$. Moreover, since $\Sigma_0$ is non-tangent to $\U_n$, it follows that $\Sigma_0\cap \U_n=\U_{n-1}$ is a non-degenerate Hermitian variety in $\Sigma_0$. Hence by induction hypothesis, $\C_{n-1}$ contains a hyperplane, say $\Pi$ of $\Sigma_0\cong \PP^{n-1}$ defined over $\Fqt$. Consequently, $\Pi$ is an $n-2$ dimensional linear subspace of $\PP^n$.

    Let us suppose that $\C_n$ contains no hyperplane $\PP^n$ defined over $\Fqt$. We denote
    $$\B(\Pi):=\{\Sigma: \Pi \subset \Sigma, \ \Sigma\ \text{is a hyperplane in $\PP^n$ defined over $\Fqt$ } \}.$$
    Note that, $|\B(\Pi)|=q^2+1$. For any $\Sigma \in \B(\Pi)$, since $\Sigma \nsubseteq \C_n$ it follows from Lemma \ref{AffineLachaud} that $|\C_n(\Fqt)\cap \U_n\cap (\Sigma\setminus \Pi)|\leq 2(q+1)q^{2n-6}$. Thus, we have
    \begin{align}
        |\C_n(\Fqt)\cap \U_n|&=|\C_n(\Fqt)\cap \U_n\cap \Pi|+\sum_{\Sigma\in \B(\Pi)}|\C_n(\Fqt)\cap \U_n\cap (\Sigma\setminus \Pi) |\nonumber\\
        &\leq |\Pi(\Fqt)\cap \U_n|+2(q^2+1)(q+1)q^{2n-6}\label{LRaffine}.
    \end{align}
   
    We now distinguish two cases:

    {\bf Case 1:} \textit{$n$ is even.} Therefore, Equation \eqref{LRaffine} entails that 
    \begin{align*}
     |\C_n(\Fqt)\cap \U_n|&\leq  \frac{q^{2n-3}+q^n-q^{n-1}-1}{q^2-1}+2(q^2+1)(q+1)q^{2n-6} \ \ \text{(Using Remark \ref{cardinality of codim two})}\\
     &= \frac{1}{q^2-1}\left[q^{2n-6}(2q^5+2q^4+q^3-2q-2)+q^{n-1}(q-1)-1\right]. 
    \end{align*}
Hence, it follows that
\begin{align*}
    B_n-|\C_n(\Fqt)\cap \U_n|&\geq q^{2n-8}(3q^5+3)+3\frac{q^{2n-5}-q^{n-1}}{q^2-1}\\
    & \ \ -\frac{1}{q^2-1}\left[q^{2n-6}(2q^5+2q^4+q^3-2q-2)+q^{n-1}(q-1)-1\right]\\
    &  =\frac{1}{q^2-1}\big[q^{2n-8}(q^7-2q^6-4q^5+5q^3+5q^2-3)-q^{n-1}(q+2)+1\big]\\
    &> 0\ \ \text{for $q\geq 7$}. 
  \end{align*}
 It's easy to check the above inequality for $n=6$, and for $n\geq 8$, the inequality follows since $(q^7-2q^6-4q^5+5q^3+5q^2-3)-(q+2)>0$. Hence $B_n > |\C_n(\Fqt)\cap \U_n|$, a contradiction to the hypothesis.

  {\bf Case 2:} \textit{$n$ is odd.} In this case $|\Pi(\Fqt)\cap \U_n|\leq |\Pi_1 \U_{n-4}(\Fqt)|=\frac{q^{2n-3}+q^{n+1}-q^n-1}{q^2-1}$. Thus, from Equation \eqref{LRaffine} we obtain 
    \begin{align*}
     |\C_n(\Fqt)\cap \U_n|&\leq  \frac{q^{2n-3}+q^{n+1}-q^n-1}{q^2-1}+2(q^2+1)(q+1)q^{2n-6}\\
    &=\frac{1}{q^2-1}\left[q^{2n-6}(2q^5+2q^4+q^3-2q-2)+q^n(q-1)-1\right]. 
    \end{align*}  
Hence, we have
\begin{align*}
  B_n-|\C_n(\Fqt)\cap \U_n|&\geq q^{2n-8}(3q^5+3)+3\frac{q^{2n-5}-q^{n-2}}{q^2-1}+q^{n-3}\\
  &\ \ -\frac{1}{q^2-1}\left[q^{2n-6}(2q^5+2q^4+q^3-2q-2)+q^n(q-1)-1\right]\\
  &=\frac{1}{q^2-1}\big[q^{2n-8}(q^7-2q^6-4q^5+5q^3+5q^2-3)\\
  & \hspace{4cm}-q^{n-3}(q^4-q^3-q^2+3q+1)+1\big]\\
  &>0.
\end{align*}
The last inequality follows since $(q^7-2q^6-4q^5+5q^3+5q^2-3)-(q^4-q^3-q^2+3q+1)>0$.
 Therefore $B_n> |\C_n(\Fqt)\cap \U_n|$, a contradiction to the hypothesis. This completes the proof.

\end{proof}

We now present the main theorem of this section, which characterizes cubic hypersurfaces in projective space under a given cardinality condition. Building on Proposition \ref{hyperplane}, this result establishes that such hypersurfaces must decompose into a union of three hyperplanes when the specified bound is exceeded.

\begin{theorem}\label{threehyp}
     Let $\C_n$ be a cubic hypersurface in $\PP^n$ defined over $\Fqt$ with $n\geq 4$ and $q\geq 7$. If $|\C_n(\Fqt)\cap \U_n|>B_n$, then $\C_n$ is the union of three  hyperplanes in $\PP^n$, each defined over $\Fqt$.
\end{theorem}
\begin{proof}
It follows from Proposition \ref{hyperplane} that $\C_n$ contains a hyperplane $\Sigma$ defined over $\Fqt$. Let us suppose that $\C_n=\Sigma\cup \Q_n$, where $\Q_n$ is a quadric hypersurface in $\PP^n$ defined over $\Fqt$. If $\Q_n$ is an irreducible quadric, then \cite[Theorem 3.3]{BBFS} entails that $|\Q_n(\Fqt)\cap \U_n|\leq A_n$. Consequently, 
\begin{align*}
   B_n< |\C_n(\Fqt)\cap \U_n|&\leq |\Sigma (\Fqt)\cap \U_n |+|\Q_n(\Fqt)\cap \U_n|\\
    &\leq |\Sigma (\Fqt)\cap \U_n |+A_n,
    \end{align*}
    which is a contradiction to Lemma \ref{ineqality1}. This completes the proof.
\end{proof}

\section{Structure of cubic hypersurfaces }\label{sec: structure}
The main goal of this section is to determine the structure of a cubic hypersurface sharing the maximum number of $\Fqt$-rational points with $\U_n$. To this end, let us first prove the following two lemmas.
\begin{lemma}\label{onenontangent}
If $\Sigma_0$ is a hyperplane in $\PP^n$ defined over $\Fqt$ that is not tangent to $\U_n$, then for any hyperplane $\Sigma \neq \Sigma_0$ in $\PP^n$ defined over $\Fqt$, the intersection $\Sigma 
 \cap \Sigma_0 \cap \U_n$ is either $\U_{n-2}$, a non-degenerate Hermitian variety of rank $n-1$, or $\Pi_0\U_{n-3}$, a degenerate Hermitian variety of rank $n-2$.    
\end{lemma}
\begin{proof}
 As $\Sigma_0\neq \Sigma$, so $\Sigma_0\cap \Sigma$ is a hyperplane in the projective space $\Sigma_0\cong \PP^{n-1}$. Moreover, since $\Sigma_0$ is non-tangent to $\U_n$, it follows from Theorem \ref{hyperplane section} (b) that $\Sigma_0\cap \U_n$ is a non-degenerate Hermitian variety contained in $\Sigma_0$. Thus, $\Sigma\cap\Sigma_0\cap \U_n=(\Sigma\cap\Sigma_0)\cap(\Sigma_0\cap \U_n)$ is a Hermitian variety either of rank $n-1$, or of rank $n-2$.     
\end{proof}
\begin{lemma}\label{bothtangent}
  Let $n\geq 4$ and $T_P(\U_n)$, $T_Q(\U_n)$ be two tangent hyperplanes to $\U_n$ at two distinct $\Fqt$-rational points $P$ and $Q$ respectively. Then the intersection $T_P(\U_n)\cap T_Q(\U_n)\cap \U_n$ is either a non-degenerate Hermitian variety $\U_{n-2}$ of rank $n-1$, or a degenerate Hermitian variety $\Pi_1\U_{n-4}$ of rank $n-3$. 
\end{lemma}
\begin{proof}
    The intersection $T_P(\U_n)\cap T_Q(\U_n)\cap \U_n$ can never be a cone $\Pi_0\U_{n-3}$. Indeed, if it were a cone with vertex at some point, say $R$,  take two lines $\ell$ and $m$ within the cone. Since $\ell, m\subseteq T_P(\U_n)\cap \U_n$, it follows that $P\in \ell\cap m$. Likewise, $Q\in \ell\cap m$. This implies $R=P=Q$, leading to a contradiction. Now, the lemma follows from Corollary \ref{codim two intersection}. 
\end{proof}

\begin{theorem}\label{evendim}
   Let $n\geq 4$ be an even integer and $\C_n$ be a cubic hypersurface in $\PP^n$ defined over $\Fqt$ with $q\geq 7$. Then $|\C_n(\Fqt)\cap \U_n|\leq 3|\U_{n-1}(\Fqt)|-2|\U_{n-2}(\Fqt)|$. Moreover, this bound is attained by a cubic hypersurface $\C_n$ in $\PP^n$, if and only if $\C_n$ is a union of three distinct non-tangent hyperplanes $\Sigma_1, \Sigma_2$ and $\Sigma_3$ defined over $\Fqt$, such that each of the hyperplanes contains a common $(n-2)$-space $\Pi_{n-2}$ defined over $\Fqt$ and $\Pi_{n-2}$ intersects $\U_n$ at a non-degenerate Hermitian variety.
\end{theorem}
\begin{proof}
 In view of Theorem \ref{threehyp}, it is enough to consider the cubic hypersurfaces, which are the union of three distinct hyperplanes defined over $\Fqt$. Let $\C_n=\Sigma_1\cup\Sigma_2\cup\Sigma_3$ be such a cubic hypersurface. We now distinguish two cases:
 
 {\bf Case 1:} \textit{$\Sigma_1,\Sigma_2$ and $\Sigma_3$ intersect in a subspace $\Pi_{n-2}$ having codimension two.} In this case, 
 \begin{equation}\label{sumeven}
 |\C_n(\Fqt)\cap \U_n|=\left(\sum_{i=1}^3 |\Sigma_i(\Fqt)\cap \U_n|\right)-2|\Pi_{n-2}(\Fqt)\cap \U_n|.
 \end{equation}
 If $\Sigma_1,\Sigma_2$ and $\Sigma_3$ are all tangents to $\U_n$, then from Remark \ref{cardinality of codim two} and Equation \eqref{sumeven} it follows 
 \begin{align}
     |\C_n(\Fqt)\cap \U_n|&\leq 3\left(q^2|\U_{n-2}(\Fqt)|+1\right)-2|\Pi_1 \U_{n-4}(\Fqt)|\nonumber \\
     &=\frac{q^{2n-3}(3q^2-2)-q^n(q-1)-1}{q^2-1}\label{sumeven1}.
 \end{align}
 Now suppose that at least one hyperplane is non-tangent to $\U_n$. Then Lemma \ref{onenontangent} and Remark \ref{cardinality of codim two} entail $|\Pi_{n-2}(\Fqt)\cap \U_n|\geq |\U_{n-2}(\Fqt)|$. Hence, from Equation \eqref{sumeven} it follows
 \begin{align}
  |\C_n(\Fqt)\cap \U_n|&\leq 3|\U_{n-1}(\Fqt)|-2|\U_{n-2}(\Fqt)|\nonumber\\
  &= \frac{q^{2n-3}(3q^2-2)+q^{n-2}(3q^2-q-2)-1}{q^2-1}.\label{sumeven2}
 \end{align}
  Comparing Equations \eqref{sumeven1} and \eqref{sumeven2}, we get that $$|\C_n(\Fqt)\cap \U_n|\leq 3|\U_{n-1}(\Fqt)|-2|\U_{n-2}(\Fqt)|, $$ and the bound is attained in this case only if all the hyperplanes are non-tangent to $\U_n$ and their intersection meets the Hermitian variety $\U_n$ at a non-degenerate Hermitian variety. 

 {\bf Case 2:} \textit{$\Sigma_1,\Sigma_2$ and $\Sigma_3$ intersect in a subspace $\Pi_{n-3}$ having codimension three.} In this case, 
 \begin{equation}\label{case2even}
   |\C_n(\Fqt)\cap \U_n|= \sum_{i=1}^3|\Sigma_i(\Fqt)\cap \U_n|-\sum_{1\leq i<j\leq 3}|\Sigma_i\cap \Sigma_j\cap \U_n(\Fqt)|+|\Pi_{n-3}(\Fqt)\cap \U_n|.
 \end{equation}
 If $\Sigma_1,\Sigma_2$ and $\Sigma_3$ are all tangents to $\U_n$, then from Equation \eqref{case2even} and Remark \ref{cardinality of codim two} we obtain
 \begin{align*}
  |\C_n(\Fqt)\cap \U_n|&
  \leq \sum_{i=1}^3 |\Sigma_i\cap \U_n(\Fqt)|-|\Sigma_1\cap\Sigma_2\cap \U_n(\Fqt)|-|\Sigma_1\cap\Sigma_3\cap \U_n(\Fqt)|\\
  &\leq 3\left(q^2|\U_{n-2}(\Fqt)|+1\right)-2|\Pi_1 \U_{n-4}(\Fqt)|\\
  &< 3|\U_{n-1}(\Fqt)|-2|\U_{n-2}(\Fqt)|.
 \end{align*}
 The last inequality has already been seen in Case 1. Now suppose that at least one hyperplane, say $\Sigma_1$, is non-tangent to $\U_n$. If one of $\Sigma_2, \Sigma_3$ is tangent, say $\Sigma_2$ is tangent then it follows from Lemma \ref{Hermitian} that $|\Sigma_2\cap \U_n(\Fqt)|<|\U_{n-1}(\Fqt)|$ and as a consequence, from Equation \eqref{case2even} we get
 \begin{align*}
    |\C_n(\Fqt)\cap \U_n|&
  \leq \sum_{i=1}^3 |\Sigma_i\cap \U_n(\Fqt)|-|\Sigma_1\cap\Sigma_2\cap \U_n(\Fqt)|-|\Sigma_1\cap\Sigma_3\cap \U_n(\Fqt)|\\
  &< 3|\U_{n-1}(\Fqt)|-2|\U_{n-2}(\Fqt)|\ \ (\text{by Lemma \ref{onenontangent} and Remark \ref{cardinality of codim two} }).
 \end{align*}
 Therefore, we may assume that all three hyperplanes are non-tangent to $\U_n$. It follows from Lemma \ref{onenontangent} that $\Sigma_i\cap \Sigma_j\cap \U_n$ is either $\U_{n-2}$ or $\Pi_0\U_{n-3}$ for $1\leq i<j\leq 3$. If $\Sigma_i\cap \Sigma_j\cap \U_n=\Pi_0\U_{n-3}$ for any pair $(i,j)\in \{(1,2),(1,3),(2,3)\}$, then 
 \begin{align*}
    |\C_n(\Fqt)\cap \U_n|&
  \leq \sum_{i=1}^3 |\Sigma_i\cap \U_n(\Fqt)|-|\Sigma_1\cap\Sigma_2\cap \U_n(\Fqt)|-|\Sigma_1\cap\Sigma_3\cap \U_n(\Fqt)|\\
  &= 3|\U_{n-1}(\Fqt)|-2|\Pi_0\U_{n-3}(\Fqt)|\\ 
  &< 3|\U_{n-1}(\Fqt)|-2|\U_{n-2}(\Fqt)|\ \ (\text{using Remark \ref{cardinality of codim two}}).
 \end{align*}
 Now suppose there exists a pair, say $(2,3)$ such that $\Sigma_2\cap \Sigma_3\cap \U_n=\U_{n-2}$. Since, $\Sigma_1\cap\Sigma_2\cap\Sigma_3$ is a hyperplane in $\Sigma_2\cap\Sigma_3$, it follows from Theorem \ref{hyperplane section} (b) that $\Sigma_1\cap\Sigma_2\cap\Sigma_3\cap \U_n=(\Sigma_1\cap\Sigma_2\cap\Sigma_3)\cap (\Sigma_2\cap \Sigma_3\cap \U_n)$ is either $\U_{n-3}$ or, $\Pi_0\U_{n-4}$. In any case, $$|\Sigma_1\cap\Sigma_2\cap\Sigma_3\cap \U_n(\Fqt)|< |\Sigma_2\cap\Sigma_3\cap \U_n(\Fqt)|=|\U_{n-2}(\Fqt)|.$$ Hence from Equation \eqref{case2even} it follows 
 \begin{align*}
    |\C_n(\Fqt)\cap \U_n|&
  < \sum_{i=1}^3 |\Sigma_i\cap \U_n(\Fqt)|-|\Sigma_1\cap\Sigma_2\cap \U_n(\Fqt)|-|\Sigma_1\cap\Sigma_3\cap \U_n(\Fqt)|\\
  &\leq 3|\U_{n-1}(\Fqt)|-2|\U_{n-2}(\Fqt)|.
 \end{align*}
  Thus, in Case 2, we always have
 $$ |\C_n(\Fqt)\cap \U_n|<3|\U_{n-1}(\Fqt)|-2|\U_{n-2}(\Fqt)|.$$ 
 Combining both cases, the assertion follows.
\end{proof}

\begin{theorem}\label{odddim}
  Let $n\geq 5$ be an odd integer and $\C_n$ be a cubic hypersurface in $\PP^n$ defined over $\Fqt$ with $q\geq 7$. Then $|\C_n(\Fqt)\cap \U_n|\leq (3q^2-2)|\U_{n-2}(\Fqt)|+3$. Moreover, this bound is attained by a cubic hypersurface $\C_n$ in $\PP^n$, if and only if $\C_n$ is a union of three distinct tangent hyperplanes $\Sigma_1, \Sigma_2$ and $\Sigma_3$ defined over $\Fqt$, such that each of the hyperplanes contains a common $(n-2)$-space $\Pi_{n-2}$ defined over $\Fqt$ and $\Pi_{n-2}$ intersects $\U_n$ at a non-degenerate Hermitian variety.   
\end{theorem}
\begin{proof}
    Let $\C_n=\Sigma_1\cup\Sigma_2\cup\Sigma_3$ be a cubic hypersurface, which is the union of three distinct hyperplanes, each defined over $\Fqt$. Let us separate two cases: 
    
     {\bf Case 1:} \textit{$\Sigma_1,\Sigma_2$ and $\Sigma_3$ intersect in a subspace $\Pi_{n-2}$ having codimension two.} In this case, 
 \begin{equation}\label{sumodd}
 |\C_n(\Fqt)\cap \U_n|=\left(\sum_{i=1}^3 |\Sigma_i\cap \U_n(\Fqt)|\right)-2|\Pi_{n-2}\cap \U_n(\Fqt)|.
 \end{equation}
Suppose $\Sigma_1,\Sigma_2$ and $\Sigma_3$ are all tangents to $\U_n$. It follows from Lemma \ref{bothtangent} and Remark \ref{cardinality of codim two} that $|\Pi_{n-2}\cap \U_n(\Fqt)|\geq |\U_{n-2}(\Fqt)|$. Hence, from Equation \eqref{sumodd} we obtain
\begin{align*}
    |\C_n(\Fqt)\cap\U_n|&\leq 3\left(q^2|\U_{n-2}(\Fqt)|+1\right)-2|\U_{n-2}(\Fqt)|\\
                        &=(3q^2-2)|\U_{n-2}(\Fqt)|+3.
\end{align*}
Let at least one of the hyperplanes, say $\Sigma_1$, be non-tangent to $\U_n$. Therefore Equation \eqref{sumodd}, Remark \ref{cardinality of codim two} and Lemma \ref{Hermitian} entail that 
\begin{align}
 |\C_n(\Fqt)\cap\U_n|&\leq |\U_{n-1}(\Fqt)|+2\left(q^2|\U_{n-2}(\Fqt)|+1\right)-2|\Pi_0\U_{n-3}(\Fqt)|\nonumber\\
                     &<(3q^2-2)|\U_{n-2}(\Fqt)|+3.\label{sumodd ineq}    
\end{align}
The inequality in \eqref{sumodd ineq} follows from a computation that 
\begin{align*}
&\left((3q^2-2)|\U_{n-2}(\Fqt)|+3\right)-\left(|\U_{n-1}(\Fqt)|+2\left(q^2|\U_{n-2}(\Fqt)|+1\right)-2|\Pi_0\U_{n-3}(\Fqt)|\right)\\
&=\frac{q^{n-2}(q^3-2q^2-q+2)}{q^2-1}>0,\ \  \text{since $q\geq 7$}.
\end{align*}

 Thus in this case $ |\C_n(\Fqt)\cap \U_n|\leq (3q^2-2)|\U_{n-2}(\Fqt)|+3$ and the bound is attained if and only if all the hyperplanes are tangents to $\U_n$ and their common space $\Pi_{n-2}$ meets the Hermitian variety $\U_n$ at a non-degenerate Hermitian variety $\U_{n-2}$.

 {\bf Case 2:} \textit{$\Sigma_1,\Sigma_2$ and $\Sigma_3$ intersect in a subspace $\Pi_{n-3}$ having codimension three.} In this case, 
 \begin{equation}\label{odd codim three}
   |\C_n(\Fqt)\cap \U_n|= \sum_{i=1}^3|\Sigma_i\cap \U_n(\Fqt)|-\sum_{1\leq i<j\leq 3}|\Sigma_i\cap \Sigma_j\cap \U_n(\Fqt)|+|\Pi_{n-3}\cap \U_n(\Fqt)|.
 \end{equation}
 Since from Equation \eqref{odd codim three} it follows $|\C_n(\Fqt)\cap \U_n|\leq \sum_{i=1}^3|\Sigma_i\cap \U_n(\Fqt)|-|\Sigma_1\cap \Sigma_2\cap \U_n(\Fqt)|-|\Sigma_1\cap \Sigma_3\cap \U_n(\Fqt)|$,
 so due to Case 1, it is enough to consider the case when all the hyperplanes are tangent to $\U_n$. Hence Lemma \ref{bothtangent} implies that $\Sigma_i\cap \Sigma_j\cap \U_n$ is either $\U_{n-2}$ or $\Pi_1 \U_{n-4}$, for all $1\leq i<j\leq 3$. If $\Sigma_i\cap \Sigma_j\cap \U_n=\Pi_1 \U_{n-4}$ for all pairs $(i,j)$, then 
 \begin{align*}
    |\C_n(\Fqt)\cap \U_n|&
  \leq \sum_{i=1}^3 |\Sigma_i\cap \U_n(\Fqt)|-|\Sigma_1\cap\Sigma_2\cap \U_n(\Fqt)|-|\Sigma_1\cap\Sigma_3\cap \U_n(\Fqt)|\\
  &= 3(q^2|\U_{n-2}(\Fqt)|+1)-2|\Pi_1 \U_{n-4}(\Fqt)|\\ 
  &<(3q^2-2)|\U_{n-2}(\Fqt)|+3 .
 \end{align*}
 The last inequality follows since $|\U_{n-2}(\Fqt)|<|\Pi_1\U_{n-4}(\Fqt)|$ (Remark \ref{cardinality of codim two}).
 
  Now suppose there exists a pair, say $(2,3)$ such that $\Sigma_2\cap \Sigma_3\cap \U_n=\U_{n-2}$. It follows that $\Sigma_1\cap\Sigma_2\cap\Sigma_3\cap \U_n=(\Sigma_1\cap\Sigma_2\cap\Sigma_3)\cap \left(\Sigma_2\cap \Sigma_3\cap \U_n\right)$ is either $\U_{n-3}$, or $\Pi_0\U_{n-4}$. In any case, $$|\Sigma_1\cap\Sigma_2\cap\Sigma_3\cap \U_n|< |\Sigma_2\cap\Sigma_3\cap \U_n|.$$ Hence from Equation \eqref{odd codim three} we get
 \begin{align*}
    |\C_n(\Fqt)\cap \U_n|&
  < \sum_{i=1}^3 |\Sigma_i\cap \U_n(\Fqt)|-|\Sigma_1\cap\Sigma_2\cap \U_n(\Fqt)|-|\Sigma_1\cap\Sigma_3\cap \U_n(\Fqt)|\\
  &\leq 3(q^2|\U_{n-2}(\Fqt)|+1)-2|\U_{n-2}(\Fqt)|\\
  &=(3q^2-2)|\U_{n-2}(\Fqt)|+3.
 \end{align*}
  Consequently, in Case 2 we always have $$ |\C_n(\Fqt)\cap \U_n|<(3q^2-2)|\U_{n-2}(\Fqt)|+3.$$ 
  This completes the proof.
\end{proof}

\section{Acknowledgments}
The author would like to thank his PhD advisor, Dr. Mrinmoy Datta, for his unwavering support and insightful discussions throughout this work. He is also grateful to the anonymous referees for their many valuable suggestions in making the article better than the previous version.


\begin{thebibliography}{GV1}

 \bibitem{BBFS}
 D. Bartoli, M. De Boeck, S. Fanali and L. Storme, On the functional codes defined by quadrics and Hermitian varieties. \emph{Des. Codes Cryptogr.} {\bf71}(2014), no.1, 21–46.

\bibitem{BC}
R.C. Bose and I.M. Chakravarti,  Hermitian varieties in a finite projective space $PG(N,q^2)$. \emph{Canad. J. Math.} {\bf 18} (1966), 1161 -- 1182.

\bibitem{BD}
P. Beelen and M. Datta,
Maximum number of points on intersection of a cubic surface and a non-degenerate hermitian surface.
\emph{Moscow Math. J.}, {\bf 20} (2020), 453--474.


\bibitem{BDH}
P. Beelen, M. Datta and M. Homma, 
A proof of sørensen’s conjecture on hermitian surfaces.
\emph{Proc. Amer. Math. Soc.} {\bf 149} (2021), 1431--1441.

\bibitem{BDMN}
P. Beelen, M. Datta, M. Montanucci, and J. Niemann. Intersection of irreducible curves and the Hermitian
curve. \emph{J. Algebra} {\bf 671} (2025), 75-94.


\bibitem{C}
 I.M. Chakravarti,  Some properties and applications of Hermitian varieties in a finite projective space $PG(N,q^2)$ in the construction of strongly regular graphs (two-class association schemes) and block designs. \emph{J. Combinatorial Theory Ser. B} {\bf 11} (1971), 268 -- 283.





\bibitem{DM}
M. Datta and S. Manna, Maximum number of points on an intersection of a cubic threefold and a non-degenerate Hermitian threefold. \emph{Finite Fields Appl.} {\bf98} 
 (2024), Paper No. 102462.



\bibitem{Di}
L. E. Dickson. \emph{Linear Groups with an Exposition of the Galois Field Theory}. Teubner, Leipzig, 1901, 312pp. (Dover, 1958)

\bibitem{E}
 F.A.B. Edoukou, Codes defined by forms of degree 2 on Hermitian surfaces and S{\o}rensen's conjecture. \emph{Finite Fields Appl.} 13 (2007), no. 3, 616 -- 627.


 \bibitem{ELX}
 F.A.B. Edoukou, S. Ling and C. Xing. Structure of functional codes defined on non-degenerate Hermitian varieties. \emph{J. Combin. Theory A} {\bf 118} 2436--2444(2011). 



\bibitem{H}
J. Harris,  Algebraic geometry. A first course. Graduate Texts in Mathematics, 133. Springer-Verlag, New York, 1992.  

\bibitem{HS}
A. Hallez and L. Storme, Functional code arising from quadric intersections with Hermitian varieties, \emph{Finite Fields Appl.} {\bf 16} (2010) 27--35.

\bibitem{HT}
J. W. P. Hirschfeld and J. A. Thas, General Galois Geometries. Mathematical Monographs. Oxford University Press, New York, 1991.




\bibitem{L} 
G. Lachaud, 
Number of points of plane sections and linear codes defined on algebraic varieties. (English summary) \emph{Arithmetic, geometry and coding theory (Luminy, 1993)}, 77 --104, de Gruyter, Berlin, 1996.


\bibitem{LR}
G. Lachaud and R. Rolland, On the number of points of algebraic sets over finite fields. \emph{J. Pure Appl. Algebra} 219 (2015), no. 11, 5117 -- 5136. 






\bibitem{SoT}
A.B. S{\o}rensen, Rational points on hypersurfaces, Reed-Muller codes and algebraic-geometric codes, Ph.D. thesis, Aarhus, Denmark, 1991.

 \bibitem{ZS}
 O. Zariski and P. Samuel,  Commutative algebra. Vol. II. Reprint of the 1960 edition. Graduate Texts in Mathematics, Vol. 29. Springer-Verlag, New York-Heidelberg, 1975. 
\end{thebibliography}
\end{document}